\documentclass[12pt]{amsart}
\usepackage[lite]{amsrefs} 
\usepackage{amssymb,amsmath,enumitem}
\usepackage{graphicx}
\newcommand{\R}{\mathbb{R}}

\newcommand{\W}{{\bf{W}}}
\newcommand{\V}{{\bf{V}}}
\newcommand{\I}{{\bf{I}}}

\numberwithin{equation}{section}
\usepackage{hyperref}
\hypersetup{colorlinks=true, citecolor=blue, linkcolor=red, urlcolor=blue}
\theoremstyle{plain}
\newtheorem{Thm}{Theorem}[section]
\newtheorem{Cor}[Thm]{Corollary}
\newtheorem{Lem}[Thm]{Lemma}

\theoremstyle{definition}

\newtheorem{Rem}[Thm]{Remark}

\theoremstyle{remark}

\begin{document}
\title[Quasilinear elliptic equations and potential theory]
{Quasilinear elliptic equations with sub-natural growth terms and \\ nonlinear potential theory}

\author{Igor E. Verbitsky}
\address{Department of Mathematics, University of Missouri,  Columbia,   \newline 
Missouri 65211, USA}
\email{\href{mailto:verbitskyi@missouri.edu}{verbitskyi@missouri.edu}}
\begin{abstract} We discuss recent  advances in the theory of 
 quasilinear  equations  of the type 
$
-\Delta_{p} u  = \sigma u^{q}    \; \; \text{in} \;\; \R^n, 
$ 
in the 
case $0<q< p-1$, where $\sigma$ is  a nonnegative measurable function, or  measure,  for the  $p$-Laplacian $\Delta_{p}u= \text{div}(|\nabla u|^{p-2}\nabla u)$, as well as 
more general quasilinear, fractional Laplacian, 
and Hessian operators.

Within this context, we obtain  
some new results, in particular, necessary and sufficient conditions for the existence of  solutions $u \in \text{BMO}(\R^n)$,  $u \in L^r_{{\rm loc}}(\R^n)$, etc., 
and 
prove an enhanced 
version of Wolff's inequality for intrinsic nonlinear potentials associated with such  problems.  
  
\end{abstract}
\dedicatory{Dedicated to Vladimir Maz'ya with affection and admiration}

\maketitle
\tableofcontents
\section{Introduction}\label{sect:intro}

We present recent advances, along with some new results,  in the existence, regularity, and nonlinear potential 
 theory associated with   the quasilinear elliptic equation  
\begin{equation}\label{eq:p-laplacian}
\begin{cases}
-\Delta_{p} u  = \sigma u^{q}, \quad u\ge 0  \quad \text{in} \;\; \R^n, \\
\liminf\limits_{x\rightarrow \infty}u(x) = 0,  
\end{cases}
\end{equation}
where $ \sigma\ge 0$ is a locally 
integrable function, or Radon measure (locally finite)  in $\R^n$,   in the \textit{sub-natural growth} case $0<q< p-1$. Such equations, together with the inhomogeneous problem    
\begin{equation}\label{eq:inhom}
\begin{cases}
-\Delta_{p} u  = \sigma u^{q} + \mu, \quad u\ge 0  \quad \text{in} \;\; \R^n, \\
\liminf\limits_{x\rightarrow \infty}u(x) = c,  
\end{cases}
\end{equation}
where $\sigma$, $\mu$ are nonnegative Radon measures, and $c\ge 0$ is a constant, have been treated in \cite{CV1}--\cite{CV3}, \cite{SV1}--\cite{SV3}, 
\cite{V3}. When $p=2$, these are sublinear elliptic equations (see \cite{BK}, \cite{QV}, 
\cite{V2}, 
and the literature cited there). 

The case $q\ge p-1$,  
which comprises Schr\"{o}dinger type equations with \textit{natural growth} terms 
when $q=p-1$, and superlinear type equations when $q>p-1$, is quite different  
(see, for example,   \cite{AP}, 
\cite{JMV}, \cite{JV}, \cite{PV1}, \cite{PV2}).

In this paper, we will be using weak solutions (possibly unbounded). 
More precisely, all solutions are understood to be  $p$-superharmonic (or equivalently, locally renormalized)  solutions   (see \cite{KKT}). We will  
assume that  
 $u\in L^{q}_{{\rm loc}}(\R^n, d\sigma)$, so that the right-hand side of \eqref{eq:p-laplacian} is a Radon measure.  
 
Among the new results obtained in  this paper are necessary and sufficient conditions on $\sigma$ for the existence of a nontrivial   solution $u \in L^{r}_{{\rm loc}}(\R^n)$ to \eqref{eq:p-laplacian} for $\frac{n(p-1)}{n-p}\le r<\infty$. Notice that  
for $0<r<\frac{n(p-1)}{n-p}$, every 
$p$-superharmonic function $u \in L^{r}_{{\rm loc}}(\R^n)$  (\cite{HKM}, \cite{MZ}). 

We will also characterize solutions   $u\in \textrm{BMO}(\R^n)$, as well as solutions in the more restricted class  
 \begin{equation}\label{eq:p-class}
\int_K |\nabla u|^p dx \le C \, \text{cap}_p (K), \quad \textrm{for all compact sets} \, \,  K \subset \R^n.
\end{equation}
 Here $\textrm{cap}_p (K)$ is the $p$-capacity 
defined by 
\begin{equation}\label{eq:p-cap}
\textrm{cap}_p (K) = \inf \left\{ \int_K |\nabla h|^p dx: \quad  h \in C^\infty_0(\R^n), \, \, 
h \ge 1 \, \, \text{on} \, \, K \right\}.
\end{equation}

We observe that in general, for the existence of a nontrivial solution $u$ to \eqref{eq:p-laplacian},  $\sigma$ must be absolutely continuous with respect to $p$-capacity, that is, 
$\sigma(K)=0$ whenever $\text{cap}_p (K)=0$. More precisely, if $u$ is a nontrivial solution  to \eqref{eq:p-laplacian}, then, 
for all compact sets 
$K \subset \mathbb{R}^n$,  we have \cite{CV2}*{Lemma 3.6}, 
\begin{equation}\label{abs-cap0}
\sigma(K)\le C\, \left[\text{cap}_p(K)\right]^{\frac{q}{p-1}}\left(\int_K u^q d\sigma\right)^{\frac{p-1-q}{p-1}}.
\end{equation}

The existence  of  solutions $u\in L^\infty(\R^n)$ 
to \eqref{eq:p-laplacian}  was characterized by Brezis and Kamin  \cite{BK} in the case $p=2$. 
They  also proved uniqueness of bounded solutions.   However, a complete  characterization of  solutions $u \in L^{r}(\R^n)$ with $r<\infty$ 
turned out to be more complicated (see \cite{V3} and the discussion below). Some sharp sufficient 
conditions for $u \in L^{r}(\R^n)$  were established recently in \cite{SV3}. See also \cite{CV1}, \cite{SV1}, 
\cite{SV2}, where finite energy solutions and their generalizations are treated.

Our main tools include certain nonlinear potentials associated with \eqref{eq:p-laplacian}.  Let $ \mathcal{M}^{+}(\R^n)$ denote the class of all (locally finite) Radon measures 
on $\R^n$. 
Given  a measure 
$\sigma \in \mathcal{M}^{+}(\R^n)$, 
 $1<r<\infty$ and 
$0<\alpha<\frac{n}{r}$, 
the Havin-Maz'ya-Wolff potential, introduced in \cite{HM},   is defined by 
\begin{equation}\label{wolff}
\W_{\alpha, r}\sigma (x) = \int_{0}^{\infty} \left[ \frac{\sigma(B(x,\rho))}{\rho^{n-\alpha r}}\right]^{\frac{1}{r-1}}\; \frac{d\rho}{\rho}, \quad x \in \R^n,
\end{equation}
where $B(x,\rho)$ is a ball centered at $x \in \R^n$ of radius $\rho >0$. 

The nonlinear potential $\W_{\alpha, r}\sigma$,  
often called 
Wolff potential, appears in harmonic analysis, approximation theory  and Sobolev spaces,  in particular spectral synthesis problems, as well as  quasilinear and fully nonlinear PDE  (see \cite{AH},   \cite{HW}, \cite{KM}, \cite{Lab}, \cite{Maz}, \cite{MZ}, \cite{PV1}). 

In the linear case $r=2$, we have $\W_{\alpha, r}\sigma = \I_{2\alpha}\sigma$  (up to a constant multiple), where the Riesz potential of order $\beta\in (0, n)$ is defined by 
\[
\I_{\beta}\sigma(x)=\int_{\R^n} \frac{d \sigma (y)}{|x-y|^{n-\beta}}, \quad x \in \R^n.  
\]

A related nonlinear potential is defined, for $1<r<\infty$, $0<\alpha<\frac{n}{r}$, by 
\begin{equation}\label{v-def}
\V_{\alpha, r}\sigma(x)  = \I_\alpha [(\I_\alpha \sigma)]^{r'-1}(x), \quad x \in \R^n.
\end{equation}
This is  the Havin-Maz'ya potential, which serves as the core notion 
 of the nonlinear potential theory 
developed 
in \cite{HM}. It is easy to see that, for all  $x \in \R^n$, 
\begin{equation}\label{v-pot}
\V_{\alpha, r}\sigma(x) \ge c (\alpha, r, n) \,  \W_{\alpha, r} \sigma(x). 
\end{equation}
The converse pointwise inequality holds only for $2 - \frac{\alpha}{n}<r<\infty$ 
(see \cite{HM}, \cite{Maz}).

Nonlinear potentials $\W_{1, p} \sigma$ ($1<p<\infty$) are intimately related to 
 the equation 
\begin{equation}\label{p-lap}
\begin{cases}
-\Delta_{p} u = \sigma, \quad u \ge 0 \quad \text{in} \;\; \mathbb{R}^n, \\ 
\liminf\limits_{x \rightarrow \infty} u(x) = 0, 
\end{cases}
\end{equation}
where $\sigma \in \mathcal{M}^{+}(\R^n)$.

The following important result is due to T.~Kilpel\"{a}inen and J.~Mal\'y 
\cite{KiMa}:  \textit{Suppose 
$u\ge 0$ is a $p$-superharmonic solution to \eqref{p-lap}. Then
\begin{equation}\label{k-m}
K^{-1}\W_{1,p} \sigma (x) \leq u(x) \leq K \W_{1,p}\sigma(x),
\end{equation} 
where $K=K(n,p)$ is a positive constant.}

It is known that  a nontrivial solution $u$  to \eqref{p-lap} 
exists if and only if 
\begin{equation}\label{k-m-cond}
\int_{1}^{\infty} \Big( \frac{\sigma(B(0,\rho))}{\rho^{n-p}}\Big)^{\frac{1}{p-1}}\frac{d \rho}{\rho}<\infty.
\end{equation} 
This is equivalent to $\W_{1,p} \sigma(x)<\infty$ for some  $x \in \R^n$, or 
equivalently quasi-everywhere (q.e.) on $\R^n$. 
In particular, \eqref{k-m-cond} may hold only in the case$1<p<n$, unless $\sigma= 0$.


The following bilateral pointwise estimates of nontrivial (minimal) solutions $u$ to 
\eqref{eq:p-laplacian} in the  case $0<q<p-1$  are fundamental 
to our approach   \cite{CV2}:
\begin{equation} \label{two-sided}
c^{-1} [(\W_{1, p} \sigma)^{\frac{p-1}{p-1-q}}+\mathbf{K}_{1, p, q}  \sigma] 
\le u \le c [(\W_{1, p} \sigma)^{\frac{p-1}{p-1-q}} +\mathbf{K}_{1, p, q}  \sigma],  
\end{equation}
where $c>0$ is a constant which  depends only on $p$, $q$, and $n$. 

Here $\mathbf{K}_{1, p, q}$ is the so-called \textit{intrinsic} nonlinear potential 
associated with \eqref{eq:p-laplacian}, which was introduced in \cite{CV2}.  
It is defined in terms of the  localized weighted norm inequalities,    
\begin{equation} \label{weight-lap}
\left(\int_{B} |\varphi|^q \, d \sigma\right)^{\frac 1 q} \le \varkappa(B) \,  ||\Delta_p \varphi||^{\frac{1}{p-1}}_{L^1(\R^n)},  
\end{equation} 
for all test functions $\varphi$ such that  $-\Delta_p \varphi \ge 0$, $\displaystyle{\liminf_{x \to \infty}} \, \varphi(x)=0$. Here $\varkappa(B)$ denotes the least constant in \eqref{weight-lap} 
associated with the measure $\sigma_B=\sigma|_B$ restricted to a ball 
$B=B(x, \rho)$.  
Then the \textit{intrinsic} nonlinear potential $\mathbf{K}_{1, p, q}$ is defined by 
\begin{equation} \label{potentialK}
\mathbf{K}_{1, p, q}  \sigma (x)  =  \int_0^{\infty} \Big(\frac{ [\varkappa(B(x, \rho))]^{\frac{q(p-1)}{p-1-q}}}{\rho^{n- p}}\Big)^{\frac{1}{p-1}}\frac{d\rho}{\rho}, \quad x \in \R^n.
\end{equation} 
As was shown in \cite{CV2},  $\mathbf{K}_{1, p, q}  \sigma \not\equiv + \infty$ if and only if 
\begin{equation}\label{suffcond1}
\int_1^{\infty} \Big(\frac{ [\varkappa(B(0, \rho))]^{\frac{q(p-1)}{p-1-q}}}{\rho^{n- p}}\Big)^{\frac{1}{p-1}}\frac{d\rho}{\rho} < \infty.  
\end{equation}

Consequently, a nontrivial $p$-superharmonic solution $u$ to  \eqref{eq:p-laplacian} exists if and only if both $\mathbf{K}_{1, p, q}  \sigma \not\equiv + \infty$ and 
$\W_{1, p} \sigma\not\equiv + \infty$, that is, 
\begin{equation}\label{suff-cond}
\int_{1}^{\infty} \Big[ \Big( \frac{\sigma(B(0,\rho))}{\rho^{n-p}}\Big)^{\frac{1}{p-1}} +  \Big(\frac{ [\varkappa(B(0, \rho))]^{\frac{q(p-1)}{p-1-q}}}{\rho^{n- p}}\Big)^{\frac{1}{p-1}}\Big]\frac{d\rho}{\rho} < \infty.  
\end{equation}

Wolff's inequality  \cite{HW}, which holds for all $1<r<\infty$,  $0<\alpha<\frac{n}{r}$, states  
\begin{equation}\label{wolff-ineq}
\mathcal{E}_{\alpha, r}[\sigma] = \int_{\R^n} (\I_\alpha \sigma)^{r'} dx \le C(\alpha, r, n) \, 
\int_{\R^n} \W_{\alpha, r}\sigma \, d \sigma,
\end{equation}
where  $r'=\frac{r}{r-1}$, and $\mathcal{E}_{\alpha, r}[\sigma]$ is the $(\alpha, r)$-energy. 
The converse inequality holds as well, since by Fubini's theorem and \eqref{v-pot}, 
\[
\int_{\R^n} (\I_\alpha \sigma)^{r'} dx = \int_{\R^n} \V_{\alpha, r} \sigma \, d \sigma \ge
 c(\alpha, r, n) \int_{\R^n} \W_{\alpha, r} \sigma \, d \sigma.  
\]

Thus, Wolff's inequality shows that, for all $1<r<\infty$,   $0<\alpha<\frac{n}{r}$, 
\begin{equation}\label{wolff-in}
\mathcal{E}_{\alpha, r}[\sigma]  \approx    \, \int_{\R^n} \W_{\alpha, r}\sigma \, d \sigma=  \int_{\R^n}\int_{0}^{\infty} \left[ \frac{\sigma(B(x,\rho))}{\rho^{n-\alpha r}}\right]^{\frac{1}{r-1}}\; \frac{d\rho}{\rho}d \sigma(x), 
\end{equation}
where the constants of equivalence depend only on $\alpha, r$, and $n$.

Several proofs of \eqref{wolff-ineq} are known, starting with the original proof 
due to Th. Wolff \cite{HW} (see also  \cite{AH}, \cite{HJ}, \cite{V1}). 
In particular, it can be deduced  
from an inequality  of Muckenhoupt and Wheeden 
for fractional integrals and  maximal functions \cite{MW} in weighted $L^r$ spaces 
(with $A_\infty$ weights). 
 A two-weight version and applications can be found in 
in \cite{COV3}, \cite{HV1}, \cite{HV2}.

It follows from  \eqref{two-sided} that 
 a necessary 
and sufficient condition for the existence of a solution $u \in L^r(\R^n)$ 
to \eqref{eq:p-laplacian} is given by: 
\begin{equation} \label{cond-two-sided}
\W_{1, p}  \sigma \in L^{\frac{r(p-1)}{p-1-q}}(\R^n) \quad \textrm{and} \quad \mathbf{K}_{1, p, q}  \sigma \in L^r(\R^n).
\end{equation}

Actually,  the first condition in  \eqref{cond-two-sided} 
is a consequence of the second one. Moreover,  the second condition
 in \eqref{cond-two-sided} can be simplified using an analogue of Wolff's inequality 
  for potentials $\mathbf{K}^d_{1, p, q}  \sigma$ \cite{V3}*{Theorem 1.1}:
\begin{equation}\label{est-wolff}
\Vert \mathbf{K}_{1, p, q}  \sigma \Vert^r_{L^r(\R^n)} \approx 
 \int_{\R^n} \int_{0}^\infty \left (\frac{ [\varkappa(B(x, \rho))]^{\frac{q(p-1)}{p-1-q}}}{\rho^{n-p}}\right)^{\frac{r}{p-1}} \frac{d \rho}{\rho} dx.
\end{equation}

In this paper, we obtain the following enhanced form of \eqref{est-wolff}.

\begin{Thm}\label{thm:main1}
Let $1<p<n$, $0<q<p-1$, $\frac{n(p-1)}{n-p}<r<\infty$, and $\sigma \in \mathcal{M}^{+}(\R^n)$. Then 
\begin{equation}\label{main-b} 
\Vert \mathbf{K}_{1, p, q}  \sigma \Vert^r_{L^r(\R^n)} \approx  \int_{\R^n} \sup_{\rho >0} \left (\frac{ [\varkappa(B(x, \rho))]^{\frac{q(p-1)}{p-1-q}}}{\rho^{n-p}}\right)^{\frac{r}{p-1}} dx,
\end{equation}
where the constants of equivalence depend only  on $p, q, r$, and $n$. 

If $n \leq p < \infty$, or $1<p<n$ 
and $0<r\le \frac{n(p-1)}{n-p}$, then  
$\mathbf{K}_{1, p, q}  \sigma \in L^r(\R^n)$ only if $\sigma=0$.  
\end{Thm}

As a corollary of Theorem \ref{thm:main1}, together with the results of \cite{V3}, we deduce that \eqref{eq:p-laplacian}  has a nontrivial solution $u \in L^r(\R^n)$ 
if and only if 
\begin{equation}\label{cond-max}
\int_{\R^n} \sup_{\rho >0} \left (\frac{ [\varkappa(B(x, \rho))]^{\frac{q(p-1)}{p-1-q}}}{\rho^{n-p}}\right)^{\frac{r}{p-1}} dx<\infty.
\end{equation}

A necessary  (but generally not sufficient) condition 
for the existence of a nontrivial solution $u \in L^{r}(\R^n)$ to \eqref{eq:p-laplacian}   follows from \eqref{main-b},  
\begin{equation}\label{cond:wolff}  
\int_{\R^n} \sup_{\rho >0} \left (\frac{\sigma(B(x, \rho))}{\rho^{n-p}}\right)^{\frac{r}{p-1-q}} dx<\infty.
\end{equation}

In fact, \eqref{cond:wolff}  is equivalent to the condition $\W_{1, p}  \sigma \in L^{\frac{r(p-1)}{p-1-q}}(\R^n) $ by 
the Muckenhoupt and Wheeden inequality  
\cite{MW} and its extensions (see \cite{HJ}, \cite{JPW}, \cite{V3}).


Using Theorem \ref{thm:main1}, we deduce the following existence results for 
equation \eqref{eq:p-laplacian}.

\begin{Thm}\label{thm:main2}
Let  $1<p<n$, $0<q<p-1$, and $\sigma \in \mathcal{M}^{+}(\R^n)$ with 
$\sigma \not\equiv 0$. Suppose that $\frac{n(p-1)}{n-p}\le r<\infty$.  Then 
there exists a nontrivial  solution $u \in L^r_{{\rm loc}}(\R^{n})$ to 
\eqref{eq:p-laplacian} if and only if condition  \eqref{suff-cond} holds, and additionally 
\begin{equation}\label{main-ex} 
\int_{B(0, R)} \sup_{0<\rho<R} \left (\frac{ [\varkappa(B(x, \rho))]^{\frac{q(p-1)}{p-1-q}}}{\rho^{n-p}}\right)^{\frac{r}{p-1}} dx<\infty,
\end{equation}
for all $R>0$. 

If $0<r <\frac{n(p-1)}{n-p}$, then there exists  a nontrivial  solution  
$u \in L^r_{{\rm loc}}(\R^{n})$ to \eqref{eq:p-laplacian}  whenever 
condition  \eqref{suff-cond} holds. 
\end{Thm}

The following corollary is deduced  from Theorem \ref{thm:main2}  under the additional assumption that there exists a constant $C=C(\sigma, p, n)$ so that 
\begin{equation}\label{cap-p}
\sigma(K)\le C\, {\rm cap}_p(K), \quad \textrm{for all compact sets} \,  \, 
K \subset \mathbb{R}^n.
\end{equation}

In the case $0<q<p-1$, condition  \eqref{cap-p} ensures  that  solutions $u$ 
to \eqref{eq:p-laplacian}  satisfy the Brezis--Kamin type pointwise estimates (\cite{BK}, \cite{CV3}): 
\begin{equation} \label{b-k}
c ^{-1} (\W_{1, p} \sigma)^{\frac{p-1}{p-1-q}} \le u \le c [(\W_{1, p} \sigma)^{\frac{p-1}{p-1-q}} + \W_{1, p} \sigma],  
\end{equation}
where $c=c (p, q, n)$ is a positive constant. 

We remark that condition  \eqref{cap-p} is also essential in the natural growth case 
$q=p-1$ (see, for instance, \cite{JMV}).  

\begin{Cor}\label{cor1} Let $1<p<n$ and $0<q<p-1$.  If 
$\sigma \in \mathcal{M}^{+}(\R^n)$ satisfies condition \eqref{cap-p},   then there exists a nontrivial solution 
$u \in L^r_{{\rm loc}}(\R^{n})$ 
to \eqref{eq:p-laplacian}, for any $0<r<\infty$,  if and only if  $\W_{1, p} \sigma\not\equiv\infty$, that is,  when \eqref{k-m-cond} holds. 
\end{Cor}

Condition  \eqref{cap-p} in Corollary \ref{cor1}   can be relaxed in a substantial way  so that estimates \eqref{b-k} still hold (see \cite{CV3}).

In the next theorem, we characterize the existence of  ${\rm BMO}$ solutions to \eqref{eq:p-laplacian}, based on Theorem \ref{thm:main1} and a 
characterization of the existence of ${\rm BMO}$ solutions to 
\eqref{p-lap} (see Lemma \ref{lemma-bmo} below). 

\begin{Thm}\label{thm:main3}
Let  $1<p<n$, $0<q<p-1$, and $\sigma \in \mathcal{M}^{+}(\R^n)$ with 
$\sigma \not\equiv 0$. If  
there exists a nontrivial  solution $u \in ${\rm BMO}$(\R^{n})$ to 
\eqref{eq:p-laplacian}, then there exists a constant  $C=C(p, q,  n)$ such that 
the following three conditions  hold:

(i) For all $x \in \R^n$ and $R>0$, 
\begin{equation}\label{bmo1} 
[\varkappa(B(x, R))]^{\frac{q(p-1)}{p-1-q}}\le C \, R^{n-p}.
\end{equation}

(ii) For all $x \in \R^n$ and $R>0$, 
\begin{equation}\label{bmo2} 
 \sigma(B(x, R))  \left (\int_R^\infty \left (\frac{[\varkappa(B(x, \rho))]^{\frac{q(p-1)}{p-1-q}}}{\rho^{n-p}}\right)^{\frac{1}{p-1}} \frac{d \rho}{\rho} \right)^q \le C \, R^{n-p}.
\end{equation}

(iii) For all $x \in \R^n$ and $R>0$, 
\begin{equation}\label{bmo3} 
 \sigma(B(x, R)) \left( \int_R^\infty \left ( \frac{\sigma(B(x, \rho))}{\rho^{n-p}}\right)^{\frac{1}{p-1}} \frac{d \rho}{\rho}\right)^{\frac{q(p-1)}{p-1-q}} \le C \, R^{n-p}.
\end{equation}

Conversely, if conditions (i), (ii), and (iii) hold, then there exists a nontrivial  solution $u \in ${\rm BMO}$(\R^{n})$ to 
\eqref{eq:p-laplacian}, provided $2-\frac{1}{n}< p<n$. When $p \ge n$, there exists  only a trivial  solution  
 to \eqref{eq:p-laplacian}.\end{Thm}

 We remark that, for $2-\frac{1}{n}< p<n$, we actually deduce   (see Lemma \ref{lemma-bmo} below) that, under 
 assumptions (i)-(iii) of Theorem \ref{thm:main3}, solutions $u$ to 
\eqref{eq:p-laplacian} satisfy 
 \begin{equation}\label{s-grad}
 \int_{B(x, R)} |\nabla u|^s dy \le C \, R^{n-s},
 \end{equation}
 for  all $0<s<p$. The restriction $p>2-\frac{1}{n}$  for this estimate can be extended to the range 
 $\frac{3n-2}{2n-1}< p\le 2-\frac{1}{n}$,  
 using recent gradient estimates obtained in \cite{NP}.  
 
 The next corollary characterizes  the existence 
  of ${\rm BMO}$ solutions,  
  for all $1<p<n$, in terms of 
 potentials $\W_{1, p} \sigma$ under assumption \eqref{cap-p}. 
 
 \begin{Cor}\label{cor2} Let   $1<p<n$ and $0<q<p-1$. Suppose 
 $\sigma\in \mathcal{M}^{+}(\R^n)$ satisfies condition \eqref{cap-p}. Then there exists a nontrivial solution 
$u \in {\rm BMO}(\R^{n})$ 
to \eqref{eq:p-laplacian} if and only if, for all $x \in \R^n$ and $R>0$, 
\begin{equation}\label{bmo4} 
 \int_{B(x, R)} (\W_{1, p}  \sigma)^{\frac{q(p-1)}{p-1-q}} d \sigma \le C \, R^{n-p}, 
\end{equation}
or, equivalently, condition \eqref{bmo3} holds. 
\end{Cor}

In a similar way, using arbitrary compact sets $K$ in place of balls $B(x, R)$ 
in \eqref{bmo4}, 
we characterize solutions $u$ to \eqref{eq:p-laplacian} in the smaller class  \eqref{eq:p-class}, which by Poincar\'{e}'s inequality is contained in ${\rm BMO} (\R^n)$.  

\begin{Thm}\label{thm:main4}
Let   $1<p<n$, $0<q<p-1$, and $\sigma \in \mathcal{M}^{+}(\R^n)$ with 
$\sigma \not\equiv 0$.  Then 
there exists a nontrivial  solution $u$  to 
\eqref{eq:p-laplacian} which satisfies condition \eqref{eq:p-class} if and only if, 
for all  compact sets $K$  in $\R^n$, 
\begin{equation}\label{class1} 
\int_{K} (\W_{1, p}  \sigma)^{\frac{q(p-1)}{p-1-q}}  d \sigma \le C \, {\rm cap}_p(K). 
\end{equation}
\end{Thm}
We remark that condition \eqref{class1} is stronger than \eqref{cap-p}. 

Our methods  are applicable to intrinsic nonlinear potentials of fractional order related to nonlinear integral equations of the type 
\begin{equation}\label{eq:int2}
u = \W_{\alpha, p}(u^{q} d\sigma) \quad \text{in} \;\; \R^{n}.
\end{equation}
Here, a solution $u\ge 0$  is understood in the sense that
$u \in L^{q}_{loc}(\R^n, \sigma)$ satisfies \eqref{eq:int2} 
$d \sigma$-a.e., or equivalently q.e. with respect to the $(\alpha, p)$-capacity (see 
\cite{AH}). In the special case $p=2$, 
this integral equation, namely  $u=\I_{2 \alpha} (u^q d \sigma)$,  is equivalent to the corresponding problem 
for the fractional Laplacian \eqref{eq:frac-laplacian} considered below. 

Bilateral pointwise estimates of solutions to \eqref{eq:int2}, similar to \eqref{two-sided}, in terms of fractional nonlinear potentials $\W_{\alpha, p} \sigma$ 
and  \textit{intrinsic} potentials $\mathbf{K}_{\alpha, p, q}$ defined in Sec. \ref{sect:wolff}  below, are obtained in \cite{CV2}. 

The following theorem is an analogue of Theorem \ref{thm:main1}. 

\begin{Thm}\label{thm:main5}
Let $1<p<\infty$,  $0<q<p-1$, $0<\alpha<\frac{n}{p}$, and $\sigma \in \mathcal{M}^{+}(\R^n)$. Suppose that $\frac{n(p-1)}{n-\alpha p}<r<\infty$.  Then 
there exists a positive  solution $u \in L^r(\R^{n})$ to \eqref{eq:int2} if and only if 
$\mathbf{K}_{\alpha, p, q}  \sigma \in L^r(\R^n)$. Moreover,  
\begin{equation}\label{main-alpha} 
\Vert \mathbf{K}_{\alpha, p, q}  \sigma \Vert^r_{L^r(\R^n)} \approx 
\int_{\R^n} \sup_{\rho >0} \left (\frac{ [\kappa(B(x, \rho))]^{\frac{q(p-1)}{p-1-q}}}{\rho^{n-\alpha p}}\right)^{\frac{r}{p-1}} dx, 
\end{equation}
where the constants of equivalence depend only  on $\alpha, p, q, r$, and $n$.

If $0<r\le \frac{n(p-1)}{n-\alpha p}$, then 
there is only a trivial supersolution $u \in L^r(\R^n)$ to \eqref{eq:int2}.
\end{Thm}

In \eqref{main-alpha}, we employ the localized embedding constants $\kappa(B)$ 
with $B=B(x, \rho)$ 
associated with certain weighted norm inequalities for potentials 
$\W_{\alpha, p}$. They are used to define the intrinsic 
potentials $\mathbf{K}_{\alpha, p, q}  \sigma$, along with their dyadic analogues 
$\mathbf{K}^d_{\alpha, p, q}  \sigma$ (see Sec. \ref{sect:wolff}). 

There are also analogues of Theorems  \ref{thm:main2}, \ref{thm:main3}, and Corollaries \ref{cor1}, \ref{cor2} for equation \eqref{eq:int2}. 
In particular, in the special case $p=2$, similar  results hold  for the fractional Laplace problem
\begin{equation} \label{eq:frac-laplacian}
\begin{cases}
\left(-\Delta \right)^{\alpha} u = \sigma u^{q}, \quad u \ge 0  \quad \text{in} \;\; \mathbb{R}^n, \\
\liminf\limits_{x \rightarrow \infty}u(x) = 0,
\end{cases}
\end{equation}
where $0<q<1$ and  $0< \alpha < \frac{n}{2}$. 

Other direct applications of Theorem  \ref{thm:main5} and related results for equation \eqref{eq:int2} in the case 
$\alpha=\frac{2k}{k+1}$, $p=k+1$ and $q<k$ involve $k$-Hessian 
equations ($k=1, 2, \ldots, n$),  based 
upon the nonlinear potential theory developed in \cite{Lab}, \cite{TW}, similar to the case $q\ge k$ considered in \cite{JV}, \cite{PV2}.

This paper is organized as follows. 
In Sec. \ref{sect:wolff}, we give definitions of  nonlinear potentials $\mathbf{K}_{\alpha, p, q}$ and discuss some of their properties. Certain lemmas on 
the existence of solutions 
$u$ to \eqref{p-lap}  in ${\rm BMO}(\R^n)$ and in the class \eqref{eq:p-class}, 
along with a dyadic version of Theorem  \ref{thm:main1}, 
are proved 
in Sec. \ref{sec3}. They are used in Sec. \ref{sec4}, where we prove Theorems  \ref{thm:main1}, \ref{thm:main2} and \ref{thm:main3}, and their analogues for equation \eqref{eq:int2}.


\section{Nonlinear potentials}\label{sect:wolff}

Let $1<p<\infty$, 
$0<\alpha< \frac{n}{p}$, and $0<q<p-1$. Let $\sigma \in \mathcal{M}^{+}(\R^n)$. We denote by $\kappa$ the least constant in 
the weighted norm inequality 
\begin{equation} \label{kap-global}
||\W_{\alpha, p} \nu||_{L^q(\R^n, d\sigma)} \le \kappa  \, \nu(\R^n)^{\frac{1}{p-1}}, \quad \forall \nu \in \mathcal{M}^{+}(\R^n).  
\end{equation}
We will also need a localized version of \eqref{kap-global} for $\sigma_E=\sigma|_E$, where $E$ is 
a Borel subset   of $\R^n$, and $\kappa(E)$ is the least constant in 
\begin{equation} \label{kap-local}
||\W_{\alpha, p}  \nu||_{L^q(d\sigma_{E})} \le \kappa (E) \, \nu(\R^n)^{\frac{1}{p-1}}, \quad \forall \nu \in \mathcal{M}^{+}(\R^n). 
\end{equation}
In applications, it will be enough to use $\kappa(E)$ 
where $E=Q$ is a dyadic cube, or $E=B$ is a ball in $\R^n$. 

It is easy to see using estimates \eqref{k-m} that embedding constants 
$\kappa(B)$ in the case $\alpha=1$ are equivalent to the constants 
$\varkappa(B)$ in \eqref{weight-lap}.

We define the \textit{intrinsic} potential of Wolff type in terms of $\kappa(B(x, \rho))$, the least constant in \eqref{kap-local} with $E=B(x, \rho)$: 
\begin{equation} \label{intrinsic-K}
\mathbf{K}_{\alpha, p, q} \sigma (x)  =  \int_0^{\infty} \left[\frac{ \kappa(B(x, \rho))^{\frac{q(p-1)}{p-1-q}}}{\rho^{n- \alpha p}}\right]^{\frac{1}{p-1}}\frac{d\rho}{\rho}, \quad x \in \R^n.  
\end{equation} 

It is easy to see that $\mathbf{K}_{\alpha, p, q} \sigma (x)  \not\equiv \infty$ if and only if 
\begin{equation} \label{inf}
 \int_a^{\infty} \left[\frac{ \kappa(B(0, \rho))^{\frac{q(p-1)}{p-1-q}}}{\rho^{n- \alpha p}}\right]^{\frac{1}{p-1}}\frac{d\rho}{\rho}< \infty,
\end{equation} 
for any (all) $a>0$. This is similar to the condition $\mathbf{W}_{\alpha, p} \sigma (x)  \not\equiv \infty$, 
which is equivalent to (see, for instance,  \cite{CV2}*{Corollary 3.2}) 
\begin{equation} \label{inf-w}
 \int_a^{\infty} \left[\frac{ \sigma(B(0, \rho))}{\rho^{n- \alpha p}}\right]^{\frac{1}{p-1}}\frac{d\rho}{\rho}< \infty.
\end{equation}

In the case of potentials $\W_{\alpha, p}$, sometimes a dyadic version $\W^d_{\alpha, p}$ of nonlinear potentials is more convenient (see \cite{HW}). In the same way, we 
find useful the dyadic version $\mathbf{K}^d_{\alpha, p, q}$ of the intrinsic potential 
$\mathbf{K}_{\alpha, p, q}$ defined by (see \cite{V3})
\begin{equation} \label{intrinsic-K-d}
\mathbf{K}^d_{\alpha, p, q} \sigma (x)  = \sum_{Q \in \mathcal{Q}}  \left[ \frac{ \kappa(Q)^{\frac{q(p-1)}{p-1-q}}} { |Q|^{1-\frac{\alpha p}{n}}}\right]^{\frac{1}{p-1}}\chi_Q(x), \quad x \in \R^n, 
\end{equation} 
where the sum is taken over all dyadic cubes (cells) $\mathcal{Q}$. 
It is easy to see that, similarly to \eqref{inf}, 
 $\mathbf{K}^d_{\alpha, p, q}\not\equiv \infty$ if and only if, for all  $P\in \mathcal{Q}$, 
\begin{equation} \label{inf-d}
 \sum_{R\supseteq P} \left[\frac{ \kappa(R)^{\frac{q(p-1)}{p-1-q}}}{|R|^{1- \frac{\alpha p}{n}}}\right]^{\frac{1}{p-1}}< \infty,
\end{equation} 
where  $R\in \mathcal{Q}$.


\section{Main lemmas}\label{sec3}

We start with some lemmas on regularity of solutions to equation \eqref{p-lap} based on certain  pointwise and integral gradient estimates (see \cite{AP}, \cite{DM}, \cite{KM}). The sufficiency 
part of the following lemma for the existence of \text{BMO} solutions 
to \eqref{p-lap}  (in bounded domains)  is known (see \cite{Mi1}*{Theorem 1.11}, 
\cite{Mi2}*{Theorem 4.3}).

\begin{Lem}\label{lemma-bmo}   Let $0< q <p-1$, and 
$2 -\frac{1}{n}<p <n$. Suppose $\sigma \in M^+(\R^n)$ satisfies the condition  
\begin{equation}\label{hausdorff}
\sigma (B(x, R))  \le C \, R^{n-p}, \quad  {\rm for \, all} \, \, x \in \R^n, \, \, R>0,
 \end{equation}
 and \eqref{k-m-cond} holds. 
Then there exists a nontrivial solution $u \in {\rm BMO} (\R^n)$ to \eqref{p-lap}. Moreover, any solution to  \eqref{p-lap} satisfies \eqref{s-grad}
  for  $0<s<p$. 

Conversely, for all  $1<p<n$, if there exists a solution $u  \in {\rm BMO} (\R^n)$  to  \eqref{p-lap}, then both conditions \eqref{k-m-cond} 
and \eqref{hausdorff} hold.  
\end{Lem}

\begin{proof} We first prove the sufficiency of condition \eqref{hausdorff} for the existence 
of a solution $u \in{\rm BMO} (\R^n)$, provided  \eqref{k-m-cond} holds, that is, 
$\W_{1, p} \sigma \not\equiv \infty$. 
The latter condition ensures (see \cite{PV2}) that there exists a solution $u$ to  \eqref{p-lap}, which satisfies pointwise bounds \eqref{k-m}. 
Next, we invoke the known pointwise gradient estimates  for 
solutions $u$ to \eqref{p-lap} in the case $2-\frac{1}{n} <p<n$, when 
$u \in W^{1,s}_{{\rm loc}}(\R^n)$ for $1\le s < \frac{n(p-1)}{n-1}$ (see \cite{DM}, \cite{KM}):
   \begin{equation}\label{p-grad}
   | \nabla u| \le C (\I_1 \sigma)^{\frac{1}{p-1}}. 
   \end{equation}
   By Poincare's inequality and \eqref{p-grad}, for $B=B(x, R)$ and $s\ge 1$, we have
   \begin{equation}\label{poincare}
\begin{split}
   \left(\frac{1}{|B|} \int_B |u- \bar u_B|^s dy\right)^{\frac{1}{s}}& \le C(n, s) \, R \, 
      \left(\frac{1}{|B|} \int_B |\nabla u|^s dy\right)^{\frac{1}{s}}\\ 
       &\le C(n, s) \, R \, \left(\frac{1}{|B|} \int_B (\I_1 \sigma)^{\frac{s}{p-1}} dy\right)^{\frac{1}{s}}.
   \end{split}
      \end{equation}
      
     We next prove that, for $1\le s < \frac{n(p-1)}{n-1}$,  
       \begin{equation}\label{est-nv}
      \int_B (\I_1 \sigma)^{\frac{s}{p-1}} dy   \le C \, |B|^{1-\frac{s}{n}},
      \end{equation}
  where $C$ does not depend on $B=B(x, R)$. Clearly,
  \[
  \I_1 \sigma(y)=(n-1) \int_0^R \frac{\sigma(B(y, \rho)}{\rho^{n-1}}\frac{d \rho}{\rho}
  + (n-1) \int_R^\infty \frac{\sigma(B(y, \rho)}{\rho^{n-1}}\frac{d \rho}{\rho}.  
  \]
  Hence, we can write  
  \[
   \int_B (\I_1 \sigma)^{\frac{s}{p-1}} dy =(n-1)^{\frac{s}{p-1}}  (I+II), 
  \]  
  where
   \begin{equation*}
\begin{split}
  I &=    
  \int_B \left( \int_0^R \frac{\sigma(B(y, \rho)}{\rho^{n-1}}\frac{d \rho}{\rho}\right)^{\frac{s}{p-1}} dy, \\
  II &=    
  \int_B \left( \int_R^\infty \frac{\sigma(B(y, \rho)}{\rho^{n-1}}\frac{d \rho}{\rho}\right)^{\frac{s}{p-1}} dy. 
  \end{split}
\end{equation*}

    By \eqref{hausdorff}, 
      \[
      \int_R^\infty \frac{\sigma(B(y, \rho)}{\rho^{n-1}}\frac{d \rho}{\rho}
      \le C \int_R^\infty \frac{\rho^{n-p}}{\rho^{n-1}}\frac{d \rho}{\rho}
      = \frac{C}{p-1} R^{1-p}.
      \] 
         
            Hence, for term $II$ we have 
      \[
      II \le C \,  R^{n-s}. 
      \]
      
      We next prove a similar estimate for term $I$ with  $s\ge1$. 
      Since $2-\frac{1}{n}<p<n$, we can assume without loss of generality that 
        \begin{equation}\label{cond-s}
      \max(1, p-1)\le s<\frac{(p-1)n}{n-1}. 
    \end{equation}

Notice that, for $y \in B(x, R)$ and $0<\rho<R$, we have $B(y, \rho)\subset B(x, 2R)$. 
Using the integral Minkowski inequality with $\frac{s}{p-1}\ge 1$ and 
taking into account \eqref{cond-s}, 
    we estimate  
              \begin{equation*}
\begin{split}
I & = \int_{B(x, R)}  \left( \int_0^R \int_{B(x, 2R)} \chi_{B(y, \rho)}(z)\, d \sigma(z)
\frac{d \rho}{\rho^n}\right)^{\frac{s}{p-1}} dy 
\\ 
& \le  \left[ \int_{B(x, 2R)} \int_0^{R} \left(\int_{B(z, \rho)} dy\right)^{\frac{p-1}{s}} \frac{d \rho}{\rho^n} d\sigma(z)\right]^{\frac{s}{p-1}} 
\\ 
& = |B(0,1)| \left[ \int_{B(x, 2R)} \int_0^{R} \rho^{\frac{n(p-1)}{s}-n} d \rho \, d\sigma(z)\right]^{\frac{s}{p-1}} 
\\ &=  C(p, s, n)  [\sigma(B(x, 2R))]^{\frac{s}{p-1}} R^{n-\frac{(n-1)s}{p-1}}.
   \end{split}
\end{equation*}
Consequently, by \eqref{hausdorff}, 
 \begin{equation*}
I \le C \, |B|^{1-\frac{s}{n}}. 
\end{equation*}
 Combining the preceding estimates for terms $I$ and $II$, we 
 obtain \eqref{est-nv}, for any ball $B=B(x, R)$. 
 
 In fact, estimate \eqref{est-nv}, and consequently \eqref{s-grad}, holds 
 for all $0<s<p$.  Indeed, by Jensen's inequality, we may assume 
 without loss of generality that $p-\epsilon \le s<p$. Then 
 by pointwise Hedberg's inequalities (see \cite{AH}*{Sec. 3.1}), 
 there exists a constant $c=c(p, n, \epsilon)$ such that, for all $\epsilon \in (0, p)$, 
 \[
 \I_1 \sigma \le c \, (M_p \mu)^{\frac{1-\epsilon}{p-\epsilon}} (\I_\epsilon \sigma)^{\frac{p-1}{p-\epsilon}},
 \]
 where  $\sigma \in \mathcal{M}^{+}(\R^n)$, and $M_p \sigma$ is the fractional maximal function of order $p$, which is uniformly bounded by \eqref{hausdorff}. Consequently, 
 by the preceding estimate and Jensen's inequality, for $p-\epsilon \le s<p$ we have
 \[
 \frac{1}{|B|} \int_B (\I_1 \sigma)^{\frac{s}{p-1}} dy   \le 
 \frac{c^{\frac{s}{p-1}}}{|B|} \int_B (\I_\epsilon \sigma)^{\frac{s}{p-\epsilon}} dy
 \le c^{\frac{s}{p-1}}  \Big(\frac{1}{|B|} \int_B (\I_\epsilon \sigma) dy\Big)^{\frac{s}{p-\epsilon}}. 
 \]
 Clearly, 
 \[
 \frac{1}{|B|} \int_B (\I_\epsilon \sigma) dy \le c(\epsilon, n) \, \Big[ \sigma (2B) R^{\epsilon-n} +\int_R^\infty \frac{\sigma(B(x, \rho))}{\rho^{n-\epsilon+1}} d \rho\Big].\]
Invoking \eqref{hausdorff}, we deduce that the right-hand side 
is bounded by $C \, R^{\epsilon-p}$, which yields  \eqref{est-nv} for all 
$0< s<p$. 
 
 Hence, by \eqref{poincare} with  
  $s\ge 1$,  we have 
\[
 \left(\frac{1}{|B|} \int_B |u- \bar u_B|^s dx\right)^{\frac{1}{s}}\le C, 
\]
where $C$ does not depend on $B$.  Thus, $u \in {\rm BMO}(\R^n)$.

Let us now prove the necessity of \eqref{hausdorff} for all $1<p<n$.  Notice that if 
a solution $u$ to  \eqref{p-lap} exists, then $\W_{1, p} \sigma \not\equiv \infty$ by \eqref{k-m}.  Suppose $u\in {\rm BMO} (\R^n)$ 
is a solution to  \eqref{p-lap}. Without loss of generality we may assume 
that $u\in W^{1,p}_{{\rm loc}}(\R^n)$. Otherwise, 
we replace $u$ with $u_k =\min (u, k)$, for  $k>0$. Since 
$u\ge 0$ is $p$-superharmonic, it follows that the same is true for $u_k$, and 
 $u_k \in W^{1,p}_{{\rm loc}}(\R^n)$ \cite{HKM}. Moreover, we clearly have $u_k \in {\rm BMO} (\R^n)$, and 
\[
 \Vert u_k \Vert_{{\rm BMO}(\R^n)}\le  \Vert u \Vert_{{\rm BMO}(\R^n)}.
 \]
 The corresponding $p$-measures $\sigma_k$ of the supersolutions $u_k$ 
 converge weakly 
 to $\sigma$ as $k \to +\infty$. Consequently, it suffices to prove \eqref{hausdorff}
 with  $\sigma_k$ and $u_k$ in place of $\sigma$ and $u$, respectively.

Let $B=B(x, R)$, and let $\eta\in C^\infty_0(\R^n)$ be a smooth cut-off 
function  supported in $2B$ such that 
$0\le \eta\le 1$, $\eta=1$ on $B$, with $|\nabla \eta|\le \frac{C}{R}$. 

We will use a Caccioppoli type estimate for supersolutions $u\ge 0$ 
to  \eqref{p-lap} on $4B$  
\cite{MZ}*{Lemma 2.113},  
which is based on the weak Harnack inequality:  
\begin{equation}\label{caccioppoli}
\int_{2B} |\nabla u|^{p-1} \eta^{p-1} |\nabla \eta| \, dy
\le C \, R^{n-p} (\inf_{B} u)^{p-1}. 
\end{equation}
In particular, by replacing $u$ in \eqref{caccioppoli} with  $u-\inf_{4B} u$, a nonnegative supersolution 
on $4B$, we deduce 
\begin{equation}\label{caccioppoli2}
\begin{split}
& \int_{2B} |\nabla u |^{p-1} \eta^{p-1} |\nabla \eta| \, dy\\
& \le C \, R^{n-p} \frac{1}{|B|}\int_{B} \Big[u(y) - \inf_{4B} u\Big]^{p-1} dy. 
\end{split}
\end{equation}
Integrating by parts and using \eqref{caccioppoli2}, we estimate
\begin{equation}\label{caccioppoli2a}
\begin{split}
\sigma (B)& =\int_{2B} \eta \, d \sigma
\\& = p \int_{2B} \eta^{p-1} \, \nabla \eta \cdot \nabla u \,   |\nabla u|^{p-2} dy 
\\ &  \le C \, R^{n-p} \frac{1}{|B|}\int_{B} \Big[u(y) - \inf_{4B} u\Big]^{p-1} dy. 
\end{split}
\end{equation}

On the other hand, if   $v$ is a 
weak subsolution on $2B$ and $s>p-1$, we have by \cite{MZ}*{Lemma 2.111}, 
\begin{equation}\label{caccioppoli3} 
\sup_{B} v 
 \le C \, \left( \frac{1}{|2 B|}\int_{2B} |v(y)|^s\, dy\right)^{\frac{1}{s}}. 
\end{equation}
Letting  $v=\bar u_{4B} -u$, 
we  obviously have 
\[
\sup_{4B} v  = \bar u_{4B} - \inf_{4B} u.  
\]
Hence,  by \eqref{caccioppoli3}  
with $4B$ in place of  $B$,  and $s>p-1$, 
\begin{equation*}
\begin{split}
0\le \bar u_{4B} - \inf_{4B} u \le   C \, \left( \frac{1}{|8B|}\int_{8B} |u- \bar u_{4B} |^s\, dy\right)^{\frac{1}{s}}.  
\end{split}
\end{equation*}
Using the well-known estimates for ${\rm BMO}$ functions, 
\[
|\bar u_{4B} - \bar u_{8B}| \le C(n) \,  \Vert u \Vert_{{\rm BMO}(\R^n)}, 
\]
we see that, for any $s>0$, 
\[
 \left(\frac{1}{|8B|}\int_{8B} |u(y) - \bar u_{4B}|^{s} dy 
  \right)^{\frac{1}{s}}
\le C \, \Vert u \Vert_{{\rm BMO}(\R^n)}.  
\]
Combining the preceding estimates, we deduce 
\begin{equation}\label{est-inf} 
\begin{split}
0\le \bar{u}_{4B} - \inf_{4B} u \le   C \, \Vert u \Vert_{{\rm BMO}(\R^n)},  
\end{split}
\end{equation}
where $C$ depends on $p, s$, and $n$. Thus, using \eqref{caccioppoli2a} together with \eqref{est-inf}, we estimate   
\begin{equation*}
\begin{split}
\sigma (B)&   \le C \, R^{n-p} \frac{1}{|B|}\int_{B} \Big[u(y) - \inf_{4B} u\Big]^{p-1} dy 
    \\& \le C \, R^{n-p} \left(\frac{1}{|B|}\int_{B} |u(y) - \bar u_{4B}|^{p-1} dy  
  +  [\bar u_{4B}  - \inf_{4B} u]^{p-1} \right)
   \\& \le C \, R^{n-p} \left(\frac{1}{|4B|}\int_{4B} |u(y) - \bar u_{4B}|^{p-1} dy 
   + \Vert u \Vert_{{\rm BMO}(\R^n)}^{p-1} 
  \right)^{p-1}
  \\& \le C \, R^{n-p} \Vert u \Vert^{p-1}_{{\rm BMO}(\R^n)}.
\end{split}
\end{equation*}
\end{proof}

\begin{Rem} An analogue of Lemma \ref{lemma-bmo} in the case $p=2$ is known for the fractional Laplacian $(-\Delta)^{\alpha}$ in place of the $p$-Laplacian. It 
can be deduced  from the fact that if $u=\I_{2\alpha} \sigma$, where $ \sigma \in \mathcal{M}^{+}(\R^n)$ and $\I_{2\alpha} \sigma\not\equiv \infty$, 
then
$u^{\sharp} \approx M_{2 \alpha} \sigma$, where $u^{\sharp} =M^{\sharp}(u)$ is the 
sharp maximal function of $u$, and 
$M_{2 \alpha}$ is the fractional maximal function of order $2 \alpha<n$; this estimate is due to D. Adams (see \cite{AH}). 
It follows that $u\in {\rm BMO} (\R^n)$ if and only if $M_{2 \alpha} \sigma\in L^\infty (\R^n)$, and $\I_{2\alpha} \sigma\not\equiv \infty$, for all $0<\alpha<\frac{n}{2}$. 
\end{Rem}

The next lemma concerns  $\sigma$ satisfying the capacity condition \eqref{cap-p}, 
which is stronger than \eqref{hausdorff}. As a result, solutions $u$ 
to \eqref{p-lap} 
belong 
to the more narrow class \eqref{eq:p-class}. Notice that this lemma 
(see \cite{AP} and the literature cited there) 
holds 
for all $1<p<n$.  In the case $2-\frac{1}{n}<p<n$ it follows 
from the pointwise gradient estimates \eqref{p-grad}. 

\begin{Lem}\label{lemma-cap} Let $0< q <p-1$, and 
$1<p <n$. Then  \eqref{p-lap} has a solution $u$ in the class \eqref{eq:p-class}
if and only if  $\sigma \in M^+(\R^n)$ satisfies  \eqref{k-m-cond} and \eqref{cap-p}.
\end{Lem} 

\begin{proof} Suppose $u$ satisfies condition \eqref{eq:p-class} and is a solution to 
 \eqref{p-lap}, so that \eqref{k-m}, and consequently \eqref{k-m-cond}, holds. 
 Let $v \in C^\infty_0(\R^n)$, $v \ge 0$, and 
 $v\ge 1 $ on a compact set $K\subset \R^n$. 
 Then, integrating by parts, we estimate 
\begin{equation*}
\begin{split}
\sigma(K) \le \int_{\R^n}  v^p d \sigma & =p \int_{\R^n} v^{p-1} \nabla v \cdot  \nabla u \, |\nabla u|^{p-2} dx\\&  \le p  \Vert \nabla v\Vert_{L^p(\R^n)} \left (\int_{\R^n} v^p  |\nabla u|^{p} dx\right)^{\frac{1}{p'}}.
\end{split}
\end{equation*}
It follows from   \eqref{eq:p-class} (see \cite{Maz}*{Sec. 2.4.1})
\[
\int_{\R^n} v^p  |\nabla u|^{p} dx\le C \int_{\R^n} |\nabla v|^{p} dx.
\]
Hence, 
\[
\sigma(K) \le C  \Vert \nabla v\Vert^p_{L^p(\R^n)}. 
\]
Minimizing the right-hand side over all such $v$ yields \eqref{cap-p}. 

To prove the converse statement, notice that there 
exists a solution $u$ to \eqref{p-lap}, in view of \eqref{k-m-cond}, which satisfies \eqref{k-m} (see, for example, \cite{PV2}). Moreover, such 
a solution is known to be unique (see \cite{KiMa}, \cite{KM}), since $\sigma$ 
is absolutely continuous with respect to the $p$-capacity by  \eqref{cap-p}. 

Clearly, \eqref{cap-p} yields \eqref{hausdorff}, that is, 
$\sigma (B(x, R))\le C \, R^{n-p}$ for all $x \in \R^n$ and $R>0$. In particular, 
$\mathbf{I}_1 \sigma \not\equiv \infty$, for all $1<p<n$. As was shown in \cite{HMV}*{Lemma 2.5}, for such $\sigma$ there exists 
a solution $v$ (not necessarily positive) to the Poisson equation $-\Delta v =\sigma$ such that 
$|\nabla v| \le C \, \mathbf{I}_1 \sigma$. Moreover, by \cite{MV}*{Theorem 2.1} 
with $l=1$ 
(see also \cite{V1}*{Theorem 1.7}), condition  \eqref{cap-p} yields that 
there exists a positive constant $c=c(p, n)$ such that, for all compact sets $K\subset \R^n$,   
 \begin{equation}\label{cap-dual}
\int_K |\nabla v|^{p'} dx \le c \,  C \, \textrm{cap}_p (K),  
\end{equation}
where $C$ is the constant in \eqref{cap-p}.  
Setting $\mathbf{F}=-\nabla v$, so that $\textrm{div} \, \mathbf{F}=\sigma$, and 
consequently $-\Delta_p u = \textrm{div} \, \mathbf{F}$, we deduce   
using 
  \cite{AP}*{Lemma 2.7} 
  that, in view of \eqref{cap-dual}, the solution  $u$ satisfies 
  \eqref{eq:p-class}. 
\end{proof}

We next prove an enhanced Wolff inequality 
for intrinsic nonlinear potentials $\mathbf{K}_{\alpha, p, q} \sigma$ 
in the dyadic case. The dyadic version $\mathbf{K}^d_{\alpha, p, q} \sigma$ 
is defined by \eqref{intrinsic-K-d}. We will also need a localized version of $\mathbf{K}^d_{\alpha, p, q} \sigma$, for a cube $P \in \mathcal{Q}$:
\[
\mathbf{K}^{d, P}_{\alpha, p, q} \sigma = \sum_{Q \subseteq P}\left[\frac{ \kappa(Q)^{\frac{q(p-1)}{p-1-q}}}{|Q|^{1- \frac{\alpha p}{n}}}\right]^{\frac{1}{p-1}} \chi_Q. 
\]

By $\mathbf{W}^{d,P}_{\alpha, p} \sigma$  and $ \mathbf{I}^{d,P}_{\alpha} \sigma$ we denote the corresponding  localized dyadic versions 
of the potentials $\mathbf{W}_{\alpha, p} \sigma$  and $ \mathbf{I}_{\alpha} \sigma$, respectively: 
  \[
\mathbf{W}^{d, P}_{\alpha, p} \sigma = \sum_{Q \subseteq P}\left[\frac{ \sigma(Q)}{|Q|^{1- \frac{\alpha p}{n}}}\right]^{\frac{1}{p-1}} \chi_Q, \quad 
 \mathbf{I}^{d, P}_{\alpha} \sigma = \sum_{Q \subseteq P} \frac{ \sigma(Q)}{|Q|^{1- \frac{\alpha}{n}}} \chi_Q.
\]

\begin{Lem}\label{lemma-new} Let $\sigma \in M^+(\R^n)$, and let $0< q <p-1$, 
$0<\alpha <\frac{n}{p}$, and 
$r>\frac{n(p-1)}{n-\alpha p}$. Then 
 \begin{equation}\label{sup-est}
\int_{\R^n} \left(\mathbf{K}^d_{\alpha, p, q} \sigma\right)^r dx \approx 
   \int_{\R^n} \sup_{P \in \mathcal{Q}: \, x \in P} \left(\frac{ \kappa(P)^{\frac{q(p-1)}{p-1-q}}}{|P|^{1- \frac{\alpha p}{n}}}\right)^{\frac{r}{p-1}}  dx,
\end{equation}
with constants of equivalence that do not depend on $\sigma$. 
\end{Lem}

\begin{proof} The lower bound in \eqref{sup-est} is obvious, since clearly 
\[
\mathbf{K}^d_{\alpha, p, q} \sigma  \ge \sup_{Q \in \mathcal{Q}: \, x \in Q} \left[\frac{ \kappa(Q)^{\frac{q(p-1)}{p-1-q}}}{|Q|^{1- \frac{\alpha p}{n}}}\right]^{\frac{1}{p-1}}.
\]

 Let us prove the upper bound. 
For $r>1$, we have   (see \cite{COV2}*{Proposition 2.2}):
  \begin{equation*} 
\int_{\R^n} \left(\mathbf{K}^d_{\alpha, p, q} \sigma\right)^r dx \approx 
\int_{\R^n} \sup_{P \in \mathcal{Q}: \, x \in P} \left(\frac{1}{|P|}  \sum_{Q \subseteq P} \left[\frac{ \kappa(Q)^{\frac{q(p-1)}{p-1-q}}}{|Q|^{1- \frac{\alpha p}{n}}}\right]^{\frac{1}{p-1}} |Q|\right)^r dx.
\end{equation*}

We will need the following estimates  (see \cite{CV2}*{Lemma 4.2 and Corollary 4.3}):  
for every $Q\subseteq P$, we have 
\begin{equation}\label{est-dd}
\begin{split}
C(\alpha, p, q, n) [\kappa(Q)]^{\frac{q}{p-1-q}}
& \le \Big[\int_Q u_P^q d \sigma\Big]^{\frac{1}{p-1}}\\ & \le \Big[\int_P u_P^q d \sigma\Big]^{\frac{1}{p-1}} \le  [\kappa(P)]^{\frac{q}{p-1-q}}.
\end{split}
\end{equation}
Here $u_P$ is a solution to \eqref{eq:p-laplacian} with $\sigma_P$ 
in place of $\sigma$.

We estimate using the lower bound in \eqref{est-dd}, 
 \begin{equation*}
\sum_{Q \subseteq P} \left[\frac{ \kappa(Q)^{\frac{q(p-1)}{p-1-q}}}{|Q|^{1- \frac{\alpha p}{n}}}\right]^{\frac{1}{p-1}} |Q|  \le \sum_{Q \subseteq P} \left[\frac{ \int_Q u^q_P d \sigma}{|Q|^{1- \frac{\alpha p}{n}}}\right]^{\frac{1}{p-1}} |Q|.
\end{equation*}

 Let $r>\frac{n(p-1)}{n-\alpha p}$. If $r\le 1$, then $p<\frac{2n}{n+\alpha}<2$. This case 
 will be considered below.

 For $p\ge 2$, we have:
  \begin{equation*}
\begin{split}  
\sum_{Q \subseteq P} \left[\frac{ \int_Q u^q_P d \sigma}{|Q|^{1- \frac{\alpha p}{n}}}\right]^{\frac{1}{p-1}} |Q| & =\int_P \mathbf{W}^{d,P}_{\alpha, p} (u^q_P d \sigma_P) dx 
\\ &  \approx \int_P \Big(\mathbf{I}^{d,P}_{\alpha p} (u^q_P d \sigma_P)\Big)^{\frac{1}{p-1}} dx
\\& \le \left( \frac{1}{|P|} \int_P \Big(\mathbf{I}^{d,P}_{\alpha p} (u^q_P d \sigma_P) \Big) dx\right)^{\frac{1}{p-1}} 
|P| 
\\& = \left( \frac{1}{|P|} \sum_{Q \subseteq P} \frac{ \int_Q u^q_P d \sigma}{|Q|^{1- \frac{\alpha p}{n}}} |Q|\right)^{\frac{1}{p-1}} 
|P|.
\end{split} 
\end{equation*}

Notice that $\sum_{Q\subseteq P} |Q|^{\frac{\alpha p}{n}} =(1-2^{-\alpha p})^{-1} 
|P|^{\frac{\alpha p}{n}}$. Consequently, we have 
  \begin{equation}\label{telescope}
\begin{split}  
 \sum_{Q \subseteq P} \frac{ \int_Q u^q_P d \sigma}{|Q|^{1- \frac{\alpha p}{n}}} |Q|&= 
  \sum_{Q \subseteq P} |Q|^{\frac{\alpha p}{n}} \int_Q u^q_P d \sigma  
 \\& = (1-2^{-\alpha p})^{-1}  |P|^{\frac{\alpha p}{n}} \int_P u^q_P d \sigma \\& \le C |P|^{\frac{\alpha p}{n}}\kappa(P)^{\frac{q(p-1)}{p-1-q}},
\end{split} 
\end{equation}
where in the last line we used the upper estimate in \eqref{est-dd}. 

Thus, in the case $p \ge 2$ and $r>1$, we have 
 \begin{equation*} 
\int_{\R^n} \left(\mathbf{K}^d_{\alpha, p, q} \sigma\right)^r dx 
\le C \int_{\R^n} \sup_{P \in \mathcal{Q}: \, x \in P} \left(\frac{ \kappa(P)^{\frac{q(p-1)}{p-1-q}}}{|P|^{1- \frac{\alpha p}{n}}}\right)^{\frac{r}{p-1}}  dx.
\end{equation*}

In the case $1<p<2$ we have $\frac{1}{p-1}>1$. Hence, clearly,  
 \begin{equation*} 
 \int_{\R^n} \left(  \sum_{Q} \left[\frac{ \kappa(Q)^{\frac{q(p-1)}{p-1-q}}}{|Q|^{1- \frac{\alpha p}{n}}}\right]^{\frac{1}{p-1}} \chi_Q\right)^r dx\le 
 \int_{\R^n} \left(  \sum_{Q} \frac{ \kappa(Q)^{\frac{q(p-1)}{p-1-q}}}{|Q|^{1- \frac{\alpha p}{n}}} \chi_Q\right)^{\frac{r}{p-1}} dx.
 \end{equation*}

Since $\frac{r}{p-1}>\frac{n}{n-\alpha p}>1$, we deduce using \cite{COV2}*{Proposition 2.2} again, 
 \begin{equation*} 
 \begin{split}
 & \int_{\R^n} \left(  \sum_{Q} \frac{ \kappa(Q)^{\frac{q(p-1)}{p-1-q}}}{|Q|^{1- \frac{\alpha p}{n}}} \chi_Q\right)^{\frac{r}{p-1}} dx\\ &\le 
  \int_{\R^n}  \sup_{P \in \mathcal{Q}: \, x \in P} \left(\frac{1}{|P|}  
   \sum_{Q \subseteq P} \frac{ \kappa(Q)^{\frac{q(p-1)}{p-1-q}}}{|Q|^{1- \frac{\alpha p}{n}}} |Q|\right)^{\frac{r}{p-1}} dx.
   \end{split}
 \end{equation*}

We estimate as above, using \eqref{telescope},
 \begin{equation*}
\begin{split}  
 & \frac{1}{|P|}   \sum_{Q \subseteq P} \frac{ \kappa(Q)^{\frac{q(p-1)}{p-1-q}}}{|Q|^{1- \frac{\alpha p}{n}}} |Q|  \le \frac{1}{|P|}  
 \sum_{Q \subseteq P} \frac{ \int_Q u^q_P d \sigma}{|Q|^{1- \frac{\alpha p}{n}}} |Q| 
 \\ &= (1-2^{-\alpha p})^{-1} \frac{ \int_P u^q_P d \sigma}{|P|^{1- \frac{\alpha p}{n}}} 
   \le C  \frac{ \kappa(P)^{\frac{q(p-1)}{p-1-q}}}{|P|^{1- \frac{\alpha p}{n}}}.
\end{split} 
\end{equation*}
Thus, as in the case $p\ge 2$ and $r>1$ above, we have 
 \begin{equation*} 
 \int_{\R^n} \left(  \sum_{Q} \frac{ \kappa(Q)^{\frac{q(p-1)}{p-1-q}}}{|Q|^{1- \frac{\alpha p}{n}}} \chi_Q\right)^{\frac{r}{p-1}} dx\le C 
  \int_{\R^n}  \sup_{P \in \mathcal{Q}: \, x \in P} \left(  
\frac{ \kappa(P)^{\frac{q(p-1)}{p-1-q}}}{|P|^{1- \frac{\alpha p}{n}}}\right)^{\frac{r}{p-1}} dx.
 \end{equation*}
\end{proof}

There is a localized version of Lemma \ref{lemma-new}. 
\begin{Lem}\label{lemma-new-loc} Let $\sigma \in M^+(\R^n)$, and let $0< q <p-1$, 
$0<\alpha <\frac{n}{p}$, and 
$r>\frac{n(p-1)}{n-\alpha p}$. Let $P \in \mathcal{Q}$. 
Then 
 \begin{equation}\label{sup-est-loc}
\int_{P} \left(\mathbf{K}^{d, P}_{\alpha, p, q} \sigma\right)^r dx \approx 
   \int_{P} \sup_{\substack{R: \, x \in R\\ R\subseteq P}}  \left(\frac{ \kappa(R)^{\frac{q(p-1)}{p-1-q}}}
   {|R|^{1- \frac{\alpha p}{n}}}\right)^{\frac{r}{p-1}}  dx,
\end{equation}
with constants of equivalence that do not depend on $\sigma$ and $P$. 
\end{Lem}

The proof of  Lemma \ref{lemma-new-loc} is essentially the same as that 
of Lemma \ref{lemma-new}, and we omit it here.


\section{Proofs of the main theorems and corollaries}\label{sec4} 

In this section, we prove the main theorems and corollaries stated in the Introduction. 

It is shown in  \cite{CV2} that \eqref{eq:p-laplacian} has a positive (super) solution 
if and only if the same is true for  \eqref{eq:int2} in the case $\alpha=1$. Moreover, 
the conditions in Theorems \ref{thm:main1} and \ref{thm:main5} are equivalent, since one can use embedding constants $\kappa(B)$ in place 
of $\varkappa(B)$ if $\alpha=1$ (see Sec. \ref{sect:wolff}). Thus, it suffices to prove only Theorem \ref{thm:main5}. 

\begin{proof}[Proof of Theorem \ref{thm:main5}]
Notice that, for all $R>0$ and $x \in \R^n$, we obviously have 
\[
 \frac{ [\kappa(B(x, R))]^{\frac{q}{p-1-q}}}{R^{\frac{n-\alpha p}{p-1}}} \le 
 2^{\frac{n-\alpha p}{p-1}} (\log 2)\int_R^{2R}  \left(\frac{ [\kappa(B(x, \rho))]^{\frac{q(p-1)}{p-1-q}}}{\rho^{n-\alpha p}}\right)^{\frac{1}{p-1}}
 \frac{d \rho}{\rho}.
\]
Hence, 
\begin{equation} \label{est-super}
\sup_{\rho>0} \frac{ [\kappa(B(x, \rho))]^{\frac{q}{p-1-q}}}{\rho^{\frac{n-\alpha p}{p-1}}} \le 
 2^{\frac{n-\alpha p}{p-1}} (\log 2) \,  \mathbf{K}_{\alpha, p, q}  \sigma(x). 
 \end{equation}
Consequently, 
\[
 \mathbf{K}_{\alpha, p, q}  \sigma\in L^r(\R^n) \Longrightarrow 
 \sup_{\rho>0} \frac{ [\kappa(B(x, \rho))]^{\frac{q}{p-1-q}}}{\rho^{\frac{n-\alpha p}{p-1}}}\in L^r(\R^n).
\]
 
It remains to prove the converse statement. 


Let $u \in L^q_{{\rm loc}}(\sigma)$ ($u\ge 0$) be a solution to  \eqref{eq:int2}. 
In \cite{CV2}, the following analogue of the bilateral pointwise estimates 
\eqref{two-sided} was obtained for nontrivial (minimal) solutions $u$ to \eqref{eq:int2}  in the  case $0<q<p-1$:
\begin{equation} \label{two-sided-frac}
c^{-1} [(\W_{\alpha, p} \sigma)^{\frac{p-1}{p-1-q}}+\mathbf{K}_{\alpha, p, q}  \sigma] 
\le u \le c [(\W_{\alpha, p} \sigma)^{\frac{p-1}{p-1-q}} +\mathbf{K}_{\alpha, p, q}  \sigma],  
\end{equation}
where $c>0$ is a constant which  depends only on $\alpha$, $p$, $q$, and $n$. 
Moreover  a nontrivial (super) solution exists if and only if both $\W_{\alpha, p} \sigma\not\equiv \infty$ and $\mathbf{K}_{\alpha, p, q}\not\equiv \infty$. 

It follows that 
$u \in L^r(\R^n)$ ($r>0$) exists if and only the following analogue of 
\eqref{cond-two-sided} holds: 
\begin{equation}\label{cond-two-sided-fr} 
 \mathbf{K}_{\alpha, p, q}  \sigma \in L^r(\R^n), \quad \W_{\alpha,p} \sigma \in L^{\frac{r(p-1)}{p-1-q}}(\R^n). 
\end{equation}
The second condition here   actually follows from the first one, both in \eqref{cond-two-sided} (in the case $\alpha=1$), and in 
\eqref{cond-two-sided-fr},  
that is, 
\begin{equation} \label{cond-two}
  \mathbf{K}_{\alpha, p, q}  \sigma \in L^r(\R^n) \Longrightarrow  \W_{\alpha, p}  \sigma \in L^{\frac{r(p-1)}{p-1-q}}(\R^n).
\end{equation}

Indeed, suppose that $\mathbf{K}_{\alpha, p, q}  \sigma \in L^r(\R^n) $. 
Using the following trivial estimate for balls $B=B(x, \rho)$, 
\begin{equation}\label{triv}
\sigma(B) |B|^{-\frac{n-\alpha p}{n(p-1)}}\le C \, [\kappa(B)]^q, 
\end{equation}
we see that 
\begin{equation}\label{low-k-est}
 \mathbf{K}_{\alpha, p, q}  \sigma(x) \ge C \, \int_0^\infty
 \Big[ \frac{\sigma(B(x, \rho))}{\rho^{n-\alpha p}}\Big]^{\frac{1}{p-1-q}} \frac{d \rho}{\rho}. 
\end{equation}
Hence,
\[
\int_0^\infty \Big[ \frac{\sigma(B(x, \rho))}{\rho^{n-\alpha p}}\Big]^{\frac{1}{p-1-q}} \frac{d \rho}{\rho} \in L^r(\R^n). 
\]

Estimates in  \cite{HJ}, \cite{JPW} yield that the preceding condition   is equivalent to $\W_{\alpha, p}  \sigma \in L^{\frac{r(p-1)}{p-1-q}}(\R^n)$. This proves \eqref{cond-two}.

In the same way, one can prove that there exists a (super) solution 
$u \in L^r(\R^n)$ to the dyadic 
version of  \eqref{eq:int2}, that is,  
\begin{equation}\label{eq:int2d}
u= \mathbf{W}^d_{\alpha, p, q} (u^q d \sigma) \quad \text{in} \, \, \R^n, 
\end{equation}
if and only if $ \mathbf{K}^d_{\alpha, p, q}  \sigma \in L^r(\R^n)$. 

It is known \cite{HW} that, for  $\omega\in \mathcal{M}^{+}(\R^n)$, 
the conditions $ \mathbf{W}^d_{\alpha, p, q}  \omega \in L^r(\R^n)$ and 
$ \mathbf{W}_{\alpha, p, q}  \omega \in L^r(\R^n)$ are equivalent.
From this it is easy to deduce, as in \cite{HW}, that  the conditions 
$ \mathbf{K}^d_{\alpha, p, q}  \sigma \in L^r(\R^n)$ and $ \mathbf{K}_{\alpha, p, q}  \sigma \in L^r(\R^n)$ are equivalent. Thus, to prove Theorem \ref{thm:main5} 
it is enough to prove its dyadic version, that is, 
to show that  
\[
\sup_{\rho>0} \frac{ [\kappa(B(x, \rho))]^{\frac{q}{p-1-q}}}{\rho^{\frac{n-\alpha p}{p-1}}}\in L^r(\R^n) \Longrightarrow \mathbf{K}^d_{\alpha, p, q}  \sigma \in L^r(\R^n).
\]

By Lemma \ref{lemma-new}, $\mathbf{K}^d_{\alpha, p, q}  \sigma \in L^r(\R^n)$ is equivalent to the right-hand side of \eqref{sup-est}, which is clearly dominated by 
its continuous version, that is, 
\begin{equation} \label{Lr-global}
\int_{\R^n} (\mathbf{K}_{\alpha, p, q}  \sigma)^r dx 
\le C \int_{\R^n} \sup_{\rho>0} \left[\frac{ \kappa(B(x, \rho))^{\frac{q(p-1)}{p-1-q}}}{\rho^{n- \alpha p}}\right]^{\frac{r}{p-1}} dx.
\end{equation}
This completes the proof of Theorem \ref{thm:main5}, and consequently Theorem \ref{thm:main1}. 
\end{proof}

\begin{proof}[Proof of Theorem \ref{thm:main2}]  In the case 
$0<r<\frac{n(p-1)}{n-p}$, it is known (\cite{HKM}, \cite{MZ}) that every $p$-superharmonic function $u \in L^r_{{\rm loc}}(\R^n)$, and so necessary and sufficient conditions for the existence 
of such a solution are given by \eqref{suff-cond}.

It is enough 
to consider solutions to \eqref{eq:int2} in the special case $\alpha=1$, although 
we present the proof for all $0<\alpha<\frac{n}{p}$. 
Let $r\ge \frac{n(p-1)}{n-\alpha p}$.  (It is easy to see that for $0<r<\frac{n(p-1)}{n-\alpha p}$, every solution  $u \in L^r_{{\rm loc}}(\R^n)$.) 
Notice that a solution 
$u \in L^r_{{\rm loc}}(\R^n)$  to \eqref{eq:int2} 
exists if and only if the following analogue of 
\eqref{cond-two-sided} holds: 
\begin{equation}\label{cond-two-sided-loc} 
 \mathbf{K}_{\alpha, p, q}  \sigma \in L^r_{{\rm loc}}(\R^n), \quad \W_{\alpha,p} \sigma \in L^{\frac{r(p-1)}{p-1-q}}_{{\rm loc}}(\R^n). 
\end{equation}
Again, as in the proof of \eqref{cond-two}, the second condition here   actually follows from the first one, that is, 
\begin{equation} \label{cond-two-loc}
  \mathbf{K}_{\alpha, p, q}  \sigma \in L^r_{{\rm loc}}(\R^n) \Longrightarrow  \W_{\alpha, p}  \sigma \in L^{\frac{r(p-1)}{p-1-q}}_{{\rm loc}}(\R^n),
\end{equation}
provided  \eqref{inf-w} and \eqref{inf} hold, which are both necessary for the existence of any solution. 
To prove \eqref{cond-two-loc}, 
let $B=B(0, R)$,  and suppose $\mathbf{K}_{\alpha, p, q} \sigma\in L^r(B)$ for every 
 $R>0$. Let us show that $\W_{\alpha, p}  \sigma \in L^{\frac{r(p-1)}{p-1-q}}(B)$.   Notice that for all $x \in B$, we have 
 \begin{equation*}
\begin{split}  
\W_{\alpha, p} \sigma_{(2B)^c}(x) & \le \int_R^\infty \Big( \frac{\sigma(B(0, 2\rho))}{\rho^{n-p}}\Big)^{\frac{1}{p-1}} \frac{d \rho}{\rho}, \\
& = 2^{\frac{\alpha p-n}{p-1}} \int_{2R}^\infty \Big( \frac{\sigma(B(0, t))}{t^{n-\alpha p}}\Big)^{\frac{1}{p-1}} \frac{d t}{t}.
 \end{split}  
 \end{equation*}
 Hence, by  \eqref{inf-w} we have $\W_{\alpha, p} \sigma_{(2B)^c}  \in L^{\infty}(B)$. It remains to show that 
 $\W_{\alpha, p} \sigma_{2B}  \in L^{\frac{r(p-1)}{p-1-q}}(B)$. In fact, we will prove 
 that $\W_{\alpha, p} \sigma_{2B}  \in L^{\frac{r(p-1)}{p-1-q}}(\R^n)$, which by Wolff's  inequality (\cite{HJ}, \cite{JPW}) is equivalent to 
 \[
 \int_{\R^n} \Big( \int_0^\infty \Big( \frac{\sigma(B(x, \rho)\cap 2B)}{\rho^{n-\alpha p}}\Big)^{\frac{1}{p-1-q}} \frac{d \rho}{\rho}\Big)^r dx<\infty.
 \]
 By \eqref{low-k-est}, we see that  $\mathbf{K}_{\alpha, p, q} \sigma_{2B}\in L^r(3B)$ yields 
\[
\int_{3B} \Big( \int_0^\infty \Big( \frac{\sigma(B(x, \rho)\cap 2B)}{\rho^{n-\alpha p}}\Big)^{\frac{1}{p-1-q}} \frac{d \rho}{\rho}\Big)^r dx<\infty.
\]
Hence, it remains   to prove that 
\[
I=\int_{(3B)^c} \Big( \int_0^\infty \Big( \frac{\sigma(B(x, \rho)\cap 2B)}{\rho^{n-\alpha p}}\Big)^{\frac{1}{p-1-q}} \frac{d \rho}{\rho}\Big)^r dx<\infty.
\]
Notice that in this integral $3R\le |x|<\rho+2R$, and consequently $\rho>\frac{|x|}{3}$.
For $r\ge \frac{n(p-1)}{n-\alpha p}$ and $0<q<p-1$ we have  $\frac{(n-\alpha p)r}{p-1-q}>n$, so that  
 \begin{equation*}
\begin{split}  
 I &\le \int_{(3B)^c}  \Big( \int_{\frac{|x|}{3}}^\infty\Big( \frac{\sigma(2B)}{\rho^{n-\alpha p}}\Big)^{\frac{1}{p-1-q}} \frac{d \rho}{\rho}\Big)^r dx 
\\&= C(\alpha, p, q,n) \, (\sigma(2B))^{\frac{r}{p-1-q}}\int_{|x|\ge 3R}  \frac{dx}{|x|^{\frac{(n-\alpha p)r}{p-1-q}}}<\infty.
 \end{split}  
 \end{equation*}
This proves \eqref{cond-two-loc}. 

It remains to show that $ \mathbf{K}_{\alpha, p, q}  \sigma \in L^r_{{\rm loc}}(\R^n)$. As above, it is enough to establish a dyadic version, 
$ \mathbf{K}^d_{\alpha, p, q}  \sigma \in L^r_{{\rm loc}}(\R^n)$. 
In other words, for any dyadic cube $P$, we need to show that 
\[
\int_P (\mathbf{K}^d_{\alpha, p, q}  \sigma)^r dx<\infty. 
\]
This condition naturally breaks into two parts: the first one is a localized condition 
\[
I=\int_P (\mathbf{K}^{d,P}_{\alpha, p, q}  \sigma_P)^r dx<\infty, 
\]
whereas the second one is 
\[
II = |P| \Big[ \sum_{R\supseteq P} \Big(\frac{ \kappa(R)^{\frac{q(p-1)}{p-1-q}}}{|R|^{1- \frac{\alpha p}{n}}}\Big)^{\frac{1}{p-1}} \Big]^{r} <\infty.
\]
By Lemma \ref{lemma-new-loc}, 
condition \eqref{main-ex} ensures that $I<\infty$, whereas $II< \infty$ by 
\eqref{inf-d}. The converse statement is obvious, since all the 
conditions \eqref{inf-w}, \eqref{inf-d}, and \eqref{main-ex}  are clearly 
necessary for the existence of a solution $u \in L^r_{{\rm loc}}(\R^n)$ in view of 
\eqref{est-super} and \eqref{cond-two-loc}. This completes the proof of Theorem \ref{thm:main2}.
\end{proof}

\begin{proof}[Proof of Corollary \ref{cor1}] We invoke estimates \eqref{b-k}, which were proved 
in \cite{CV3} under the assumption \eqref{cap-p}. Since $\frac{p-1}{p-1-q}>1$, by H\"older's inequality it is enough to ensure that $(\mathbf{W}_{1, p} \sigma)^{\frac{p-1}{p-1-q}} \in L^r_{{\rm loc}}(\R^n)$. As above, it suffices to show that, for any dyadic cube $P$, 
\[
\int_P (\mathbf{W}^d_{1, p}  \sigma)^{\frac{r(p-1)}{p-1-q}}  dx=I + II<\infty, 
\]
where 
\[
I=\int_P (\mathbf{W}^{d, P}_{1, p}  \sigma)^{\frac{r(p-1)}{p-1-q}}  dx,  \quad  
II = |P| \Big[ \sum_{R\supseteq P} \left(\frac{ \sigma(R)}{|R|^{1- \frac{p}{n}}} 
\right)^{\frac{1}{p-1}}\Big]^{\frac{r(p-1)}{p-1-q}}.
\]
The second term is finite by the necessary condition \eqref{k-m-cond}, 
which ensures that $\mathbf{W}_{1, p} \sigma\not\equiv\infty$. 

To show that the localized term $I<\infty$, notice that by a localized version of Wolff's  
inequality (see, for instance, \cite{V3}), 
\[
I \approx \sum_{Q\subseteq P} \Big(\frac{ \sigma(Q)}{|Q|^{1- \frac{p}{n}}} 
\Big)^{\frac{r}{p-1-q}}|Q|.
\]
On the other hand, \eqref{cap-p} 
yields the estimate (\cite{CV3}*{Lemma 2.1 and Remark 2.2}) 
\[
\int_P (\mathbf{W}^{d, P}_{1, p}  \sigma)^{s}  d\sigma\le C \sigma(P)<\infty, 
\]
for any $s>0$. In particular, for $s\ge 1$ we obviously have 
\[
\sum_{Q\subseteq P} \Big(\frac{ \sigma(Q)}{|Q|^{1- \frac{p}{n}}} 
\Big)^{\frac{s}{p-1}}\sigma(Q)\le C \sigma(P)<\infty. 
\]
Setting $s=\frac{[r-(p-1-q)](p-1)}{p-1-q}$, where without loss of generality we may assume 
$s\ge 1$ (for $r$ large enough), we deduce 
\[
\sum_{Q\subseteq P} \Big(\frac{ \sigma(Q)}{|Q|^{1- \frac{p}{n}}} 
\Big)^{\frac{r}{p-1-q}}|Q|\le |P|^{\frac{p}{n}}\sum_{Q\subseteq P}\Big(\frac{ \sigma(Q)}{|Q|^{1- \frac{p}{n}}} 
\Big)^{\frac{s}{p-1}}\sigma(Q)<\infty.
\]
Hence, $II<\infty$ as well, so that $(\mathbf{W}_{1, p} \sigma)^{\frac{p-1}{p-1-q}} \in L^r_{{\rm loc}}(\R^n)$. It follows by \cite{CV3}*{Theorem 1.2} that there exists  
a nontrivial solution $u \in L^r_{{\rm loc}}(\R^n)$.
\end{proof}

\begin{proof}[Proof of Theorem \ref{thm:main3}] By Lemma \ref{lemma-bmo} 
with $d \omega=u^q d \sigma$ in place of $\sigma$, a solution $u \in {\rm BMO}(\R^n)$ to \eqref{eq:p-laplacian} exists if (in the case $2-\frac{1}{n}<p<n$) and only if, for every ball $B(x, R)\subset \R^n$,  
 \begin{equation}\label{BMO-cond}
\omega(B(x, R)) = \int_{B(x, R)} u^q d \sigma \le C \, R^{n-p}. 
 \end{equation} 
Moreover, if $p>2-\frac{1}{n}$, then such a solution actually satisfies \eqref{s-grad} 
for all $0<s<\frac{n(p-1)}{n-1}$. 

By estimates  \eqref{two-sided}, it follows that \eqref{BMO-cond} holds if and only if 
 \begin{equation}\label{split-eq}
 \int_{B(x, R)}[ (\mathbf{W}_{1, p} \sigma)^{\frac{q(p-1)}{p-1-q}} 
 +  (\mathbf{K}_{1, p} \sigma)^{q}] d \sigma \le C \, R^{n-p}.
 \end{equation} 
 
 Moreover, by \cite{CV2}*{Lemma 4.2},  
for every ball $B=B(x, R)$, we have 
\begin{equation}\label{est-d}
[\kappa(B)]^{\frac{q(p-1)}{p-1-q}}
 \le C(p, q, n) \, \int_B u^q d \sigma.
\end{equation}
This proves the necessity of condition \eqref{bmo1}. 
To prove the necessity of condition  \eqref{bmo2}, notice that, for all $y \in B(x, R)$ 
and $\rho>2R$, 
we have $B(y, \rho)\supset B(x, \frac{\rho}{2})$. Letting 
$t= \frac{\rho}{2}$, we estimate,  for $y \in B(x, R)$,
 \begin{equation*}
\begin{split}  
\mathbf{K}_{1, p} \sigma(y) & \ge  \int_{2R}^\infty  
\Big(\frac{ [\kappa(B(x, \frac{\rho}{2}))]^{\frac{q(p-1)}{p-1-q}}}{\rho^{n-p}} \Big)^{\frac{1}{p-1}}
\frac{d \rho}{\rho}\\& =2^{\frac{p-n}{p-1}} \int_{R}^\infty  
\Big(\frac{ [\kappa(B(x, t))]^{\frac{q(p-1)}{p-1-q}}}{t^{n-p}} \Big)^{\frac{1}{p-1}}
\frac{d t}{t}.
 \end{split}  
 \end{equation*}
Thus,  \eqref{bmo2} follows from 
\eqref{split-eq}. The necessity of  \eqref{bmo3}  is deduced in a similar way. 

To prove the sufficiency of  conditions \eqref{bmo1}, \eqref{bmo2} and   \eqref{bmo3},  we first verify the estimate of  
the localized term in \eqref{split-eq}, with $\sigma_{2B}$ in place of $\sigma$ (here 
$2B=B(x, 2R)$), that is, 
\begin{equation}\label{est-d-loc}
\int_{B} [(\W_{1, p} \sigma_{2B})^{\frac{q(p-1)}{p-1-q}}+(\mathbf{K}_{1, p, q}  \sigma_{2B})^q] d \sigma  \le C \, R^{n-p}. 
\end{equation}

We invoke  the estimate \cite{CV2}*{Corollary 4.3},  
\begin{equation}\label{est-db}
\int_{2B} u_{2B}^q d \sigma \le   
[\kappa(2B)]^{\frac{q(p-1)}{p-1-q}}. 
\end{equation}
Here $u_{2B}$ denotes  a nontrivial solution to \eqref{eq:p-laplacian} with $\sigma_{2B}$ in place 
 of $\sigma$. Combining \eqref{est-db} 
with the lower pointwise estimate \eqref{two-sided} for $u_{2B}$ in place of $u$, namely, 
\[
c(p, q, n) [(\W_{1, p} \sigma_{2B})^{\frac{p-1}{p-1-q}}+\mathbf{K}_{1, p, q}  \sigma_{2B}] 
\le u_{2B}, 
\]
together with \eqref{bmo1}, yields \eqref{est-d-loc}. 

To obtain similar estimates  for  $\sigma_{(2B)^c}$ (the portion of $\sigma$ supported 
outside $2B$) in place of $\sigma$ in \eqref{split-eq}, notice that, for all 
$y \in B(x, R)$, we have  $[B(x, 2R)]^c\cap B(y, \rho)=\emptyset$ if $0<\rho<R$,  
and $B(y, \rho)\subset B(x, 2 \rho)$ if $\rho>R$. Hence, for $y \in B=B(x, R)$, 
 \begin{equation*}
\begin{split}  
\W_{1, p} \sigma_{(2B)^c}(y) & \le \int_R^\infty \Big( \frac{\sigma(B(x, 2\rho))}{\rho^{n-p}}\Big)^{\frac{1}{p-1}} \frac{d \rho}{\rho}, \\
\mathbf{K}_{1, p} \sigma_{(2B)^c}(y) & \le \int_R^\infty  \Big(\frac{ [\kappa(B(x, 2\rho))]^{\frac{q(p-1)}{p-1-q}}}{\rho^{n-p}}\Big)^{\frac{1}{p-1}} \frac{d \rho}{\rho}.
 \end{split}  
 \end{equation*}
Letting $t=2 \rho$ in these integrals, we estimate 
  \begin{equation*}
\begin{split}  
& \int_{B} [(\W_{1, p} \sigma_{(2B)^c})^{\frac{q(p-1)}{p-1-q}}+(\mathbf{K}_{1, p, q}  \sigma_{(2B)^c})^q] d \sigma\\ 
& \le C \sigma(B) \Big(\int_{2R}^\infty \Big( \frac{\sigma(B(x, t))}{t^{n-p}}\Big)^{\frac{1}{p-1}} \frac{d t}{t}
\Big)^{\frac{q(p-1)}{p-1-q}}
\\ &+ C \sigma(B) \Big( \int_{2R}^\infty  \Big(\frac{ [\kappa(B(x, t))]^{\frac{q(p-1)}{p-1-q}}}{t^{n-p}}\Big)^{\frac{1}{p-1}} \frac{d t}{t}\Big)^q.
 \end{split}  
 \end{equation*}
Using conditions \eqref{bmo2} and   \eqref{bmo3}, we deduce 
\[
\int_{B} [(\W_{1, p} \sigma_{(2B)^c})^{\frac{q(p-1)}{p-1-q}}+(\mathbf{K}_{1, p, q}  \sigma_{(2B)^c})^q] d \sigma\le C \, R^{n-p}. 
\]
This completes the proof of \eqref{split-eq}, and consequently, 
Theorem \ref{thm:main3}.\end{proof} 

\begin{proof}[Proof of Corollary \ref{cor2}] As in the proof of Theorem  \ref{thm:main3}, it follows from Lemma \ref{lemma-bmo} 
that a nontrivial solution $u \in {\rm BMO}(\R^n)$ to \eqref{eq:p-laplacian} exists if and only if, for every ball $B(x, R)\subset \R^n$,  condition \eqref{BMO-cond} holds.

 Moreover, 
 the upper estimate in \eqref{b-k}, which holds 
 for a minimal solution $u$ under the assumption \eqref{cap-p}, yields that  a sufficient condition for $u \in {\rm BMO}(\R^n)$ is given by 
   \begin{equation}\label{BMO-2}
 \int_{B(x, R)} [(\W_{1, p} \sigma)^{\frac{q(p-1)}{p-1-q}}+( \W_{1, p} \sigma)^q] d \sigma\le C \, R^{n-p}. 
  \end{equation}
 Since by \eqref{cap-p} we have  $\sigma (B(x, R))\le C \,  R^{n-p}$ for any ball $B(x, R)$, 
 it follows by H\"older's inequality that we can drop the second term in 
 \eqref{BMO-2}. In other words, the condition 
  \begin{equation}\label{BMO-3}
 \int_{B} (\W_{1, p} \sigma)^{\frac{q(p-1)}{p-1-q}} d \sigma\le C \, R^{n-p}, 
  \end{equation}
 for all balls $B=B(x, R)$, is sufficient. It is also necessary, since it follows 
 from  \eqref{BMO-cond} and the lower estimate   \eqref{b-k}.
 
 It remains to show that \eqref{BMO-3} is equivalent to  \eqref{bmo3}. Clearly, for all 
 $y \in B(x, R)$, we have  
 \[
\W_{1, p} \sigma(y) \ge C  \, \int_{2R}^\infty \Big( \frac{\sigma(B(x, t))}{t^{n-p}}\Big)^{\frac{1}{p-1}} \frac{d t}{t}. 
 \]
 Hence, \eqref{BMO-3}$\Longrightarrow$\eqref{bmo3}. To prove the converse, 
it suffices  to estimate only the localized part of \eqref{BMO-3}, namely, 
  \begin{equation}\label{BMO-4}
 \int_{B} (\W_{1, p} \sigma_{2B})^{\frac{q(p-1)}{p-1-q}}  d \sigma \le C \, R^{n-p}, 
  \end{equation} 
  since the term corresponding to $(\sigma_{2B})^c$ is estimated 
  as above using \eqref{bmo3}. Invoking  again \cite{CV3}*{Lemma 2.1 and Remark 2.2} 
  with $s=\frac{q(p-1)}{p-1-q}$ we see that \eqref{cap-p} yields 
  \[
   \int_{B} (\W_{1, p} \sigma_{2B})^{\frac{q(p-1)}{p-1-q}}  d \sigma \le C \, \sigma (2B) \le C \, \, R^{n-p}. 
  \]
  This shows that \eqref{BMO-4} holds, that is, \eqref{bmo3}$\Longrightarrow$ \eqref{BMO-3}.
 \end{proof}

The proof of Theorem \ref{thm:main4}, based on Lemma   
\ref{lemma-cap}, is similar to the above arguments, and is omitted here.



\end{document}